\documentclass[11pt]{amsart}

\usepackage{pb-diagram}

\usepackage{amsfonts}
\usepackage{amsmath}
\usepackage{amssymb}

\newcommand{\dbar}{\overline{\partial}}

\newcommand{\ddt}[1]{\frac{\partial #1}{\partial t}}
\newcommand{\R}{\mathcal{R}}

\newcommand{\ddbar}{\sqrt{-1}\partial\dbar}

\newtheorem{theorem}{Theorem}[section]
\newtheorem{proposition}{Proposition}[section]
\newtheorem{lemma}{Lemma}[section]

\newtheorem{definition}{Definition}[section]
\newtheorem{corollary}{Corollary}[section]

\begin{document}

\address{Department of Mathematics, Rutgers University, Piscataway, NJ 08854}

\email{jiansong@math.rutgers.edu}

\thanks{Research supported in
part by National Science Foundation grants DMS-0847524.}

\centerline{ {\bf \large RIEMANNIAN GEOMETRY  OF K\"AHLER-EINSTEIN CURRENTS II}   \footnote{Research supported in part by National Science Foundation
grant  DMS-1406124} }

\bigskip

\centerline{  \small AN ANALYTIC PROOF OF KAWAMATA'S BASE POINT FREE THEOREM }

\bigskip
\bigskip

\centerline{ \small  JIAN SONG }

\bigskip

\bigskip

{\noindent \small A{\scriptsize BSTRACT}. $~$~~It is proved by Kawamata that the canonical bundle of a projective manifold is semi-ample if it is big and nef. We give an analytic proof using the Ricci flow, degeneration of Riemannian manifolds and  $L^2$-theory. Combined with our earlier results, we construct unique singular  K\"ahler-Einstein metrics with a  global Riemannian structure on canonical models. Our approach can be viewed as the Kodaira embedding theorem on singular metric spaces with canonical K\"ahler metrics. }

{\footnotesize \tableofcontents}

\section{Introduction}

This is a sequel to our earlier work \cite{S3}. A well-known theorem of Kawamata \cite{K1, K2, KMM, KM} states that if the canonical bundle $K_X$ of a projective manifold $X$ is big and nef, then it must be semi-ample, i.e., the linear system $|mK_X|$ is base point free for some sufficiently large $m\in \mathbb{Z}$. This is a very important result and has many deep generalizations and applications in the minimal model program. In particular, it implies the finite generation of the canonical ring  and the abundance conjecture for minimal models of general type.

Recent progress in K\"ahler geometry has revealed deep connections and interplay among nonlinear PDEs, Riemannian geometry and complex algebraic geometry. The solution to the Yau-Tian-Donaldson conjecture \cite{Y2, T1, T2, D, DS, CDS1, CDS2, CDS3, T4} has established the relation between the existence of K\"ahler-Einstein metrics and the $K$-stability for Fano manifolds, using the theory for degeneration of Riemannian K\"ahler manifolds by Cheeger-Colding \cite{CC1, CC2} and Cheeger-Colding-Tian \cite{CCT}, and Hormander's $L^2$-estimates on singular metric spaces to establish the partial $C^0$-estimate proposed by Tian \cite{T1}. The analytic minimal model program with Ricci flow proposed by the author and  Tian \cite{ST1, ST2, ST3} connects finite time singularity  of the K\"ahler-Ricci flow to geometric and birational surgeries,  and its  long time behavior to the existence of singular K\"ahler-Einstein metrics and the abundance conjecture. There have been many results in this direction \cite{SW1, SW2, SW3, SoYu, S3}.  It is further proposed by the author \cite{S3} that the K\"ahler-Ricci flow should give a global uniformization in terms of K\"ahler-Einstein metrics for projective varieties as well as a local uniformization in terms of the transition of shrinking and expanding solitons for singularities arising simultaneously from the K\"ahler-Ricci flow and birational transformation \cite{S1, S2}.

The singular K\"ahler-Einstein metrics on projective manifolds of general type are first constructed by Tsuji \cite{Ts} by the K\"ahler-Ricci flow. This is generalized for projective varieties with log terminal singularities of general type and projective Calabi-Yau varieties with log terminal singularities  by Eyssidieux-Guedj-Zeriahi \cite{EGZ} (see also \cite{Z}) based on the fundamental work of Kolodziej \cite{Kol1} in the study of complex Monge-Ampere equations with singular data.  Recently, it is shown by the author \cite{S3} that such K\"ahler-Einstein currents also admit both global and local Riemannian structures on the canonical models of smooth minimal models of general type and projective Calabi-Yau varieties admitting a crepant resolution of singularities.  However, the above results all assume the abundance conjecture for the minimal models of general type.    Our main result is to remove this assumption and to give an analytic and Riemannian geometric proof for the following base point free theorem of Kawamata. 
 
\begin{theorem}\label{main1} Let $X$ be a projective manifold. If  the canonical bundle $K_X$ is big and nef, then it is semi-ample, i.e., $m K_X$ is globally generated for some sufficiently large $m\in \mathbb{Z}^+$. 

\end{theorem}

Projective manifolds of big and nef canonical bundle are called minimal models of general type. Immediately, one can conclude the canonical ring of $X$, a smooth minimal model of general type,  is finitely generated and $X$ admits a unique canonical model $X_{can}$ birationally equivalent to $X$,  from the pluricanonical map $\Phi=|mK_X|: X \rightarrow X_{can}$ for sufficiently large $m$.  Applying the result of Theorem 1.2 in \cite{S3}, we have the following Riemannian geometric counterpart of Theorem \ref{main1}. 
\begin{corollary} 
Let $X$ be an $n$-dimensional smooth minimal model of general type. Then there exists a unique smooth K\"ahler-Einstein metric $g_{KE}$ on $X_{can}^\circ$, the smooth part of the canonical model $X_{can}$ for $X$,  satisfying
\begin{enumerate}

\medskip

\item $g_{KE}$ extends uniquely to a K\"ahler current $\omega_{KE} \in c_1(X_{can})$ on $X_{can}$ with bounded local potentials,

\medskip

\item the metric completion of $(X_{can}^\circ, g_{KE})$ is a compact metric length space $(X_\infty, d_\infty)$ homeomorphic to the projective variety $X_{can}$ itself, 

\medskip

\item the singular set $\mathcal{S}$ of $(X_\infty, d_\infty)$ has Hausdorff dimension no great than $n-4$. In particular,   $(X_\infty \setminus \mathcal{S}, d_\infty)$ is convex in $(X_\infty, d_\infty)$ and it is isomorphic to $(X_{can}^\circ, g_{KE})$. 

\end{enumerate}

\end{corollary}

Our method is based on the scheme developed in our earlier work \cite{S3}. We apply estimates from the K\"ahler-Ricci flow, pluripotential theory for degenerate complex Monge-Ampere equations \cite{Kol1, EGZ}, the theory for degeneration of Riemannian manifolds, Hormander's $L^2$-theory and a variation of Tian's partial $C^0$-estimates using the $H$-condition of Donaldson-Sun \cite{DS} (see also \cite{T3} for a different approach).  In the course of proving Theorem \ref{main1}, we also achieve the $L^\infty$-estimate for local potential of the K\"ahler-Einstein current on the minimal manifold of general type without assuming Kawamata's theorem.

The proof for Theorem \ref{main1} can  be easily adapted to prove the following theorem also due to Kawamata \cite{K1, K2}.

\begin{theorem}\label{main2} Let $X$ be a projective manifold with $c_1(X)=0$. If a holomorphic line bundle $L$ over $X$ is big and nef, then it is semi-ample.  

\end{theorem}

Similarly, we have the following corollary from Theorem  \ref{main2} and Theorem 1.1 in \cite{S3}.

\begin{corollary} 
Let $X$ be an $n$-dimensional projective manifold with $c_1(X)=0$. Then for any big and nef line bundle $L$ on $X$, the linear system $|mL|$ for sufficiently large $m$ gives a birational morphism $\Phi: X \rightarrow Y$ from $X$ to a unique projective variety $Y$ with canonical singularities and $c_1(Y)=0$. 

Furthermore,  there exists a unique smooth Ricci-flat K\"ahler metric $g_{CY}$ on $Y^\circ$, the smooth part of  $Y$,  satisfying
\begin{enumerate}

\medskip

\item $g_{CY}$ extends uniquely to a K\"ahler current $\omega_{CY} \in c_1(
\Phi_*L)$ on $Y$ with bounded local potentials , 

\medskip

\item the metric completion of $(Y^\circ, g_{CY})$ is a compact metric length space $(Y_\infty, d_\infty)$ homeomorphic to the projective variety $Y$ itself, 

\medskip

\item the singular set $\mathcal{S}$ of $(Y_\infty, d_\infty)$ has Hausdorff dimension no great than $n-4$. In particular, $(Y_\infty \setminus \mathcal{S}, d_\infty)$ is convex in $(Y_\infty, d_\infty)$  is isomorphic to $(Y^\circ, g_{CY})$. 

\end{enumerate}

\end{corollary}

Our approach follows the traditional and more constructive  proof for the Kodaira embedding theorem by Hormander's $L^2$-estimates without applying any sophisticated results from algebraic geometry such as the non-vanishing theorem.  The canonical singular K\"ahler-Einstein metric on the minimal model of general type $X$ plays an important role in  both applying the analytic $L^2$-estimates and proving that it coincides with the metric space from the degeneration of the Riemannian almost K\"ahler-Einstein metrics on $X$. Therefore our method can be viewed as the Kodaira embedding theorem on singular metric spaces with canonical Riemmanian K\"ahler metrics. We believe that it can be applied to prove the general base point free theorem of Kawamata for any big and nef divisor on a smooth projective manifold using the K\"ahler-Einstein metric with conical singularities. We also hope that our approach can  lead to an analytic and Riemannian geometric proof for the finite generation of canonical rings on smooth varieties of general type, which is already proved by algebraic methods \cite{BCHM, Siu2}. 

In general,  if $X$ is a projective manifold of  positive Kodaira dimension, it admits a unique canonical twisted K\"ahler-Einstein current constructed in \cite{ST1, ST2} (also see \cite{To2} for collapsing Calabi-Yau manifolds). We hope such analytic canonical metrics can be used to prove the abundance conjecture if the Riemannian collapsing theory for K\"ahler manifolds or the K\"ahler-Ricci flow can be established. 

%%%%%%%%%%%%%%%%%%%%%%%%%%%%%%%%%%%%%%%%%%%%%%%%%%%%%%%

\section{A priori estimates for the K\"ahler-Ricci flow}

In this section, we will establish some basic estimates for the singular K\"ahler-Einstein metrics on smooth minimal projective manifolds of general type. Let $X$ be a minimal manifold of general type of complex dimension $n$. Let $\Omega$ be a smooth volume form on $X$ and let $\chi = \ddbar\log \Omega \in -c_1(X)$. For any smooth  K\"ahler form $\omega_0\in H^{1,1}(X, \mathbb{R})\cap H^2(X, \mathbb{Q})$, we consider the following Monge-Ampere flow
\begin{equation}\label{maflow}
 \ddt{\varphi} = \log \frac{ (\chi + e^{-t}(\omega_0 - \chi) + \ddbar \varphi)^n}{\Omega} - \varphi, ~ \varphi(0)=0.
 \end{equation}
Without loss of generality, we assume that $\omega_0- \chi$ is K\"ahler.
Let $\omega(t) = \chi + e^{-t} (\omega_0 - \chi) + \ddbar \varphi(t)$ and $g(t)$ be the associated K\"ahler metrics. Then $g(t)$ solves the normalized K\"ahler-Ricci flow
\begin{equation}\label{krflow}
\ddt{g} = - Ric(g) - g, ~ g(0) = g_0,
\end{equation}
where $g_0$ is the initial  K\"ahler metric associated to the K\"ahler form $\omega_0$. 

Let $h_\chi$ be a smooth hermitian metric on $K_X$ defined by $h_\chi = \Omega^{-1}$. Since $K_X$ is big and nef, by Kodaira's lemma there exists an effective divisor $D$ on $X$ such that there exists $\epsilon_0>0$ such that $K_X - \epsilon D$ is ample for all   $\epsilon \in (0, \epsilon_0)$.  Therefore for any $\epsilon$ sufficiently small, there exists a smooth hermitian metric $h_{D, \epsilon}$ such that $\chi - \epsilon Ric(h_{D, \epsilon})$ is K\"ahler. We let $\sigma_D$ be the defining section of $D$ and fix a smooth hermitian metric $h_D$ on $D$.

\begin{lemma} \label{prop21} The following hold for the parabolic Monge-Ampere equation (\ref{maflow}). 

\begin{enumerate}

\item  There exists $C>0$ such that for all $t>0$, we have on $X$
$$\varphi\leq C, ~ \ddt\varphi \leq C. $$

\item For any $\epsilon>0$, there exists $C_\epsilon>0$ such that for all $t\geq 0$, we have on $X$
$$ \varphi \geq \epsilon \log |\sigma_D|^2_{h_D} - C_\epsilon.$$

\medskip

\item There exist $\lambda, C>0$ such that for all $t\geq 0$, we have on $X$
$$tr_{\omega_0}(\omega(t)) \leq C |\sigma_D|_{h_D}^{-2\lambda} . $$

\end{enumerate}

\end{lemma}

\begin{proof}   The first statement follows immediately from the maximum principle. The second and the third statement follow from Tsuji's tricks by applying the maximum principle to $ \varphi - \epsilon \log |\sigma_D|^2_{h_{D, \epsilon}} $ for any $\epsilon \in (0, \epsilon_0]$ and $\log tr_{\omega_0} (\omega) - A \left( A \varphi - \log |\sigma_D|^2_{h_{D, 1/A}} \right) $ for some fixed sufficiently large $A>0$.

\end{proof}

The following lemma follows from the standard third order estimates (either by local estimates \cite{SW} or by global estimates \cite{PSeS} with weights ) and local higher order estimates. 

\begin{lemma} \label{lm21}

For any $k>0$ and compact set $K \subset \subset X \setminus D$, there exists $C_{k, K}>0$ such that for all $t>0$, 
$$ \| \varphi \|_{C^k(K)} \leq C_{k,K}. $$

\end{lemma}

\begin{lemma} \label{lm22}

$\ddt{\varphi}$ converges smoothly to $0$ on any compact subset of $X\setminus D$. 

\end{lemma}

\begin{proof} We apply a trick of Zhang \cite{Z2} by looking at the evolution of the following quantity  
$$\Box_t ~ \left( (e^t-1) \ddt{\varphi}  - \varphi\right) = n - tr_\omega(\omega_0), $$
where $\Box_t ~= \ddt{} - \Delta_t$ and $\Delta_t$ is the Laplacian associated to $g(t)$. 
Therefore there exist $C_1, C_2>0$ such that for $t\geq 1$, 
$$\ddt{\varphi} \leq \frac{ nt + \varphi + C_1}{e^t -1} \leq \frac{ C_2~e^{-t/2}}{2}. $$
and so 
$$\ddt{ \left( \varphi + C_2~e^{-t/2} \right) } \leq 0. $$
This implies that $\varphi+ C_2 ~e^{-t/2}$ decrease to $\varphi_\infty \in PSH(X, \chi)\cap C^\infty(X\setminus D)$ and so $\ddt\varphi$ must tend to $0$ away from $D$.

\end{proof}

The following corollary immediately from from Lemma \ref{prop21}, Lemma \ref{lm21} and the proof of Lemma \ref{lm22}.

\begin{corollary} \label{co21}

The solution  $\varphi(t)$ of the parabolic Monge-Ampere equation (\ref{maflow}) converges to a unique $\varphi_\infty\in PSH(X, \chi) \cap C^\infty(X\setminus D)$ as $t\rightarrow \infty$. In particular, for any $\epsilon>0$, there exists $C_\epsilon>0$ such that on $X$
$$ \varphi_\infty \geq  \epsilon \log |\sigma_D|^2_{h_D} - C_\epsilon. $$

\end{corollary}

Let $\omega_\infty = \chi+ \ddbar \varphi_\infty$. Then $\varphi_\infty$ satisfies the K\"ahler-Einstein equation on  $X\setminus D$  
\begin{equation}\label{keq}
\omega_\infty^n = (\chi + \ddbar \varphi_\infty)^n = e^ {\varphi_\infty}~ \Omega, ~ Ric(\omega_\infty) = - \omega_\infty. 
\end{equation}
 We also have the following existence and uniqueness result.

\begin{lemma} \label{lm23} There exists a unique $\phi  \in PSH(X, \chi) \cap C^\infty(X\setminus D)$ such that

\begin{enumerate}

\item  $(\chi+\ddbar \phi)^n = e^\phi ~\Omega$ on $X\setminus D$, 
\medskip

\item for any $\epsilon>0$, there exists $C_\epsilon>0$ such that on $X$
$$ \phi \geq  \epsilon \log |\sigma_D|^2_{h_D} - C_\epsilon. $$

\end{enumerate}

\end{lemma}

\begin{proof} It suffices to prove the uniqueness. Suppose there exists another solution $\varphi'$. Then we consider $ \psi_\epsilon= \varphi(t) - \varphi' - \epsilon e^{-t} \log |\sigma_D|^2_{h_D} + A e^{-t}$ for sufficiently small $\epsilon>0$. Then
$$\ddt ~ \psi_\epsilon = \log \frac{ ( \chi + \ddbar \varphi' + e^{-t} (\omega_0 - \chi - \epsilon Ric(h_D)) + \ddbar \psi_\epsilon )^n}{ (\chi +\ddbar \varphi' )^n} - \psi_\epsilon. $$
Since $\psi_\epsilon$ tends to $\infty$ along $D$ and $\psi_\epsilon(0) \geq 0$ for sufficiently large $A>0$, we can apply the maximum principle and so $\psi_\epsilon \geq 0$ for all $t\geq 0$. By letting $t\rightarrow \infty$, we have 
$$\varphi_\infty \geq \varphi'. $$
Then it immediately follows from the comparison principle that $\varphi_\infty = \varphi'$ on $X\setminus D$ and so the lemma follows. 

\end{proof}

We let $h_t = (\omega(t))^{-n}$ be the hermitian metric on $K_X$ for $t\in [0, \infty)$. Then we have the following lemma.

\begin{lemma} \label{lm24}  For any $\sigma \in H^0(X, mK_X)$, there exits $C>0$ such that for all $t\geq 0$, 
\begin{equation} \label{pw1}
 \sup_X |  \sigma |^2_{h_t^m} \leq C.
\end{equation}

\end{lemma}

\begin{proof} Without loss of generality, we can assume that for sufficiently large $m \in \mathbb{Z}$, a basis $\{\sigma_j\}_{j=0}^{d_m} $ of $H^0(X, mK_X)$ gives a birational map from $X$ into the projective space $\mathbb{CP}^{d_m}$, where $d_m+1 = h^0(X, mK_X)$. We consider a resolution for the base locus $\{\sigma_j\}_j$
$$\pi: X' \rightarrow X$$ such that
$$ \pi^*(mK_X ) = L +E, $$
where $L$ is semi-ample and $E$ is the fixed part of $\pi^*(mK_X)$. We can assume that  $E$ is a divisor of simple normal crossings. Since $L$ is big and semi-ample, there exists an effective divisor $D'$ on $X'$ such $L - \epsilon D'$ is ample for all $\epsilon>0$. 

 The closed form $\theta =m^{-1} \ddbar \log (\sum_{j=0}^{d_m} |\sigma_j|^2)|$ on $X'\setminus E$ is the Fubini-Study metric which smoothly extends to $X'$ globally in  $c_1(L)$.  There exists a smooth hermitian metric $h_{D'}$ on the line bundle associated to $[D']$ such that 
$$\theta- \epsilon Ric(h_{D'})>0$$
for all sufficiently small $\epsilon>0$.  

Let $\Omega_m = ( \sum_{j=0}^{d_m}  |\sigma_j|^2 )^{1/m}$ be the smooth real nonnegative $(n, n)$-form on $X$. We let 
$$H = \log \frac{ \Omega_m}{\Omega_t}, $$
where  $\Omega_t = \omega(t)^n $. 
Then $H$ is bounded above and smooth outside the base locus of $\{\sigma_j\}_j$ and the evolution equation for $H$  is given by 
$$ \Box_t  ~H = n  - tr_\omega (\theta). $$
We now lift the above equation to $X'$ and it is smooth on $X'\setminus E. $

We now consider the Monge-Ampere equation
$$(\frac{\theta}{2} + \ddbar \phi)^n =  \Omega'$$
for some smooth volume form $\Omega'$ on $X'$ with $\int_{X'} \Omega' = (2)^{-n} \int_{X'} \theta^n$. By standard argument \cite{Ts, EGZ, ST2}, $\phi\in C^0(X')\cap PSH(X', \theta)\cap C^\infty(X'\setminus D)$. 
Let 
$$G_\epsilon = H +  \phi +  \epsilon \log |\sigma_{D'}|^2_{h_{D'}}   + \epsilon^2 \log |\sigma_E|^2_{h_E}, $$
where $\sigma_E$ is the defining section of $E$ and $h_E$ is a smooth hermitian metric on the line bundle associated to $E$.  Then $G_\epsilon $ is smooth on $X'\setminus (E\cup D)$ and tends to $-\infty$ along $E\cup D$ and we can apply the maximum priniciple for $G_\epsilon$ on $X'\setminus ( E\cup D)$. 
Then there exists $C>0$ such that 
\begin{eqnarray*}
 \Box_t  ~G_\epsilon &=& n - tr_\omega \left( \frac{\theta}{2} + \ddbar \phi \right) - tr_\omega \left( \frac{\theta}{2} - \epsilon Ric(h_D) - \epsilon^2 Ric(h_E)  \right) \\
  &\leq& n -  \left( \frac{ \Omega'}{\Omega_t} \right)^{1/n} \leq  n -  C\left( \frac{ \Omega_m}{\Omega_t} \right)^{1/n} 
  \end{eqnarray*}
for sufficiently small $\epsilon>0$.
By the maximum principle, $G_\epsilon$ is bounded above uniformly for all $t$ and $\epsilon$.  By letting $\epsilon \rightarrow 0$, $H$ is uniformly bounded above and this proves the lemma.

\end{proof}

\begin{corollary} \label{co22} Let $h_\infty = h_\chi e^{-\varphi_\infty } =\left( \omega_\infty^{n} \right)^{-1}$. 
Then for any $m$ and $\sigma\in H^0(X, m K_X )$, there exists $C>0$ such that 
\begin{equation} \label{pw3} 
\sup_X |\sigma|^2_{ (h_\infty)^m} <C, 
\end{equation}
or equivalently there exists $C>0$ such that on $X$, 
$$\varphi_\infty \geq  m^{-1} \log |\sigma|^2_{h_\chi^m} - C. $$

\end{corollary}

\begin{lemma} \label{lm25} Let $h_t = \left( \omega(t)^n \right)^{-1}$ be the smooth hermitian metric on $K_X$. For any $\sigma\in H^0(X, mK_X)$, there exists $C>0$ such that for all $t$, 
$$ \sup_X | \nabla_t ~\sigma |^2_{g(t), h_t^m} \leq C. $$
\end{lemma}
\begin{proof} For simplicity, we write $|\sigma|^2$ and $| \nabla _t ~\sigma |^2$ for $|\sigma|^2_{h_t^m}$ and $| \nabla _t ~\sigma |^2_{g(t), h_t^m}$. The evolution equation for $| \nabla _t ~\sigma |^2$ is given by 
\begin{equation}\label{pw2} \Box_t  ~|\nabla_t ~ \sigma|^2      =  - |\nabla _t \nabla_t ~\sigma|^2 + (1+mn)  |\nabla_t ~ \sigma |^2  . 
\end{equation}
Also we have
$$ \Box_t  ~| \sigma|^2      =   m n | \sigma|^2  - |\sigma|^{-2} |\nabla_t~ |\sigma|^2|^2 =   m n | \sigma|^2  - |\nabla_t~ \sigma|^2. $$
Let $H = |\nabla_t ~ \sigma|^2  + A |\sigma|^2$. The lemma is then proved by applying the maximum principle to $H$ after choosing sufficiently large $A>0$.

\end{proof}

By Lemma \ref{lm25} and the local smooth convergence of $\omega(t)$ on $X\setminus D$, we have the following corollary. 

\begin{corollary} \label{co23}   For any $m$ and $\sigma\in H^0(X, mK_X)$, there exists $C>0$ such that 
\begin{equation} \label{pw4}
 \sup_X |\nabla_\infty \sigma|^2_{g_\infty, (h_\infty)^m}  <C.
 \end{equation}

\end{corollary}

We remark that the constant $C$ in (\ref{pw1}, \ref{pw3}, \ref{pw2}, \ref{pw4}) depends on $m$ and $\sigma$.

\begin{definition} \label{def21}

Let $\mathcal{R}_X$ be the set of all points $p$ on $X$ such that all $\mu$-jets at $p$ are globally generated by some power of $K_X$ for $|\mu| \leq 2$, where $\mu=(\mu_1, ..., \mu_n)\in \mathbb{Z}^n$ is nonnegative. 

\end{definition}

In local holomorphic coordinates $z$ with $p=0$, the $\mu $-jets at $p$ are given by $\prod_{i =1}^n z_i^{\mu_i}$.

\begin{lemma} \label{lm26}

 $\mathcal{R}_X$ is open in $X$ and $\varphi(t)$ converges smoothly to $\varphi_\infty$ on $\mathcal{R}_X$ as $t\rightarrow \infty$. In particular, $\varphi_\infty \in C^\infty(\mathcal{R}_X)$.

\end{lemma}

\begin{proof} Obviously, $\mathcal{R}_X$ is open. Let $p\in \mathcal{R}_X$. Then there exist $m>0$ and a basis $\{ \sigma_j \}_{j=0}^{d_m} $ of $ H^0(X, mK_X)$ such that $\{ \sigma_j\}_j$ gives a local embedding in a small neighborhood of $p$ into a projective space $\mathbb{CP}^{d_m}$. 
Let $\theta$ be the pullback of the Fubini-Study metric. 
First we note that 
$$\Box_t~ \dot{\varphi} = -e^{-t} tr_{\omega} (\omega_0 - \chi) - \dot \varphi \leq - \dot \varphi.$$
Therefore by maximum principle, $\dot\varphi$ is uniformly bounded above. 
Then we consider
$$H = \varphi + \dot\varphi - m^{-1} \log (\sum_{j=0}^{d_m} |\sigma_j|^2_{h_\chi^m}) .$$
Then outside of the common base locus of $\{ \sigma_j \}_j$, 
$$ \Box_t~ H = tr_\omega(\theta)\geq 0. $$
Since $H =\infty$ along the base locus of $\{\sigma_j\}_j$, from the maximum principle, $H$ is uniformly bounded below. 
Now let 
$$G = \log tr_\omega(\theta) - A H. $$
Then for sufficiently large $A$, applying the argument for the the parabolic Schwarz lemma in \cite{ST1, ST2}, we have 
$$\Box_t~ G \leq  -tr_\omega(\theta) + C$$
 outside  the base locus of $\{\sigma_j\}_j$. We also note that for sufficiently large $A>0$, $G$ is smooth outside  the base locus of $\{\sigma_j\}_j$ and tends to $-\infty$ along the base locus of $\{\sigma_j\}_j$.  At the maximum point, $tr_\omega(\theta) \leq C$ and so $G$ is uniformly bounded above. This implies that
$$tr_\omega(\theta) \leq C e^{AH} \leq C' (\sum_{j=0}^{d_m}  |\sigma_j|^2_{h_\chi^m})^{A/m}. $$
On the other hand, $\omega^n$ is bounded above uniformly  away from the base locus of the linear system $|mK_X|$ for fixed $m$. Also $\theta^n$ is uniformly equivalent to $\Omega$ near $p$. 
Therefore $\omega$ is equivalent to $\theta$ near $p$. The third and higher order local regularity near $p$ is achieved by standard argument and this completes the proof of the lemma.

\end{proof}

\section{Riemannian geometric limits}

In this section, we will apply the Cheeger-Colding theory \cite{CC1, CC2, CCT} for degeneration of Riemannian manifolds with Ricci curvature bounded below,  the work of Tian-Wang \cite{TW} for almost K\"ahler-Einstein metrics, and a local $L^2$-estimates to study the Riemannian structure of $(\mathcal{R}_X, g_\infty)$ and its metric  completion.

We first pick a K\"ahler form 
\begin{equation}\label{refmc}
\omega_0' = \chi - \epsilon_0 Ric(h_{D, \epsilon_0})
\end{equation}
 for some sufficiently small $\epsilon_0>0$.  We now consider the following family of Monge-Ampere equations for $k\in \mathbb{Z}^+$.
\begin{equation}\label{maeq}
( (1-e^{-k} )\chi + e^{-k}\omega_0' + (1-e^{-k})\ddbar \varphi_k)^n = e^{\varphi_k} \Omega.
\end{equation}
Let 
$\omega_k = \chi + \ddbar \varphi_k + (1-e^{-k} )^{-1} e^{-k} \omega_0'  $ and $g_k$ be the corresponding K\"ahler metric. Then the  curvature equation for  $g_k$ is given by 
\begin{equation}\label{curveq}
Ric(g_k) = -  g_k + (1-e^{-k})^{-1}e^{-k} g_0'.
\end{equation} 

The following estimates follow by similar estimates from Section 2, using elliptic argument instead of parabolic estimates. 

\begin{lemma} \label{lm31} We have the following uniform estimates.

\begin{enumerate}

\item  There exists $C>0$  such that for all $k>0$, , we have on $X$
$$\sup_X \varphi_k\leq C, $$

\item for any $\epsilon>0$, there exists $C_\epsilon>0$ such that for all $k >0$, we have on $X$
$$ \varphi_k \geq \epsilon \log |\sigma_D|^2_{h_D} - C_\epsilon,$$
where $h_D$ is a fixed smooth hermitian metric on $D$, 
\medskip

\item there exist $\lambda, C>0$ such that for all $k>0$, we have on $X$
$$tr_{\omega_0}(\omega(t)) \leq C |\sigma_D|_{h_D}^{-2\lambda} . $$

\item for any $l>0$ and compact set $K \subset \subset X \setminus D$, there exists $C_{l, K}>0$ such that for all $k>0$, 
$$ \| \varphi \|_{C^l(K)} \leq C_{l,K}. $$

\end{enumerate}

\end{lemma}

From Lemma \ref{lm31}, $\varphi_k$ converges to some $\varphi'_\infty\in PSH(X, \chi)\cap C^\infty(X\setminus D)$ solving $(\chi+ \ddbar \varphi'_\infty)^n = e^{\varphi'_\infty} \Omega$. In particular, by the uniqueness from Lemma \ref{lm23}, $\varphi'_\infty =\varphi_\infty$, where $\varphi_\infty$ is the limiting potential from the Monge-Ampere flow (\ref{maflow}).

We will now verify in the following lemma for the almost K\"ahler-Einstein condition introduced in \cite{TW}. 
\begin{lemma} \label{locesk} Let $g_k $ be the solution of equation (\ref{maeq}) for $k\in \mathbb{Z}^+$. Then $g_k$ satisfies the following almost K\"ahler-Einstein conditions.

\begin{enumerate}

\item $Ric(g_k) \geq -g_k$, 
\smallskip

\item there exists $p\in X \setminus D$ and $r_0, \kappa >0$ such that for all $k$, 
$$B_{g_k}(p, r_0)\subset\subset X\setminus E, ~Vol(B_{g_k}(p, r_0) ) \geq \kappa,$$ 

\item  Let $g_k(t)$ be the solution of the normalized K\"ahler-Ricci flow 
$$\ddt{g_k(t)} = -Ric(g_k(t)) - g_k(t), ~ g_k(0) = g_k. $$ Then 
$$ \lim_{k \rightarrow \infty} \int_0^1 \int_{X} \left| R(g_k(t)) + n \right| dV_{g_k(t)} dt =0. $$

\end{enumerate}

\end{lemma}

\begin{proof} (1) and (2) follow easily from equation (\ref{curveq}) and Lemma \ref{lm31}. Notice that the minimum of the scalar curvature is non decreasing along the Ricci flow while $R(g_k(0)) > -n$. Therefore
\begin{eqnarray*}
&&\int_0^1\int_{X} |R(g_k(t))+n|dV_{g_k(t)}dt \\
 &\leq & \int_{X} \int_0^1( R(g_k(t))+n) dV_{g_k(t)} dt \\
&=& - \epsilon_0\int_0^1 e^{-t-k} (1-e^{-k})^{-1} [D'] \cdot \left( [\chi] -\epsilon_0 e^{-t-k} (1-e^{-k})^{-1}[D'] \right)^{n-1} dt 
\end{eqnarray*}
converges to $0$ as $k\rightarrow \infty$.

\end{proof}

We then apply the main results of Tian-Wang \cite{TW} to obtain the following proposition. 

\begin{proposition} \label{ghlimgt} Let $(X, p, g_k)$ be the almost K\"ahler-Einstein manifolds in Lemma \ref{locesk}. Then $(X, p, g_k)$ converges to  a metric length space $(X_\infty, p_\infty, d_\infty)$ satisfying 

\begin{enumerate}

\item $\R$, the regular set of $X_\infty$, is a smooth open dense convex set in $X_\infty$, 

\smallskip

\item the limiting metric $d_\infty$ induces a  smooth K\"ahler-Einstein metric $g_{KE}$ on $\R$ satisfying $Ric(g_{KE}  ) = - g_{KE} $,  

\smallskip

\item the singular set $\mathcal{S}$ has Hausdorff dimension no greater than $2n-4$.

\end{enumerate}

\end{proposition} 

The rest of the section is to prove that the regular part $\mathcal{R}$ coincides with $\mathcal{R}_X$ and $g_{KE}$ coincides with $g_\infty$, the limiting K\"ahler-Einstein metric from the K\"ahler-Ricci flow.

\begin{definition}

Let $\mathcal{S}_X$ be the set of points $q_\infty $ in $X_\infty$ such that there exist a sequence of points $q_k\in (X\setminus \mathcal{R}_X, g_k)$ converging to $q_\infty$ in Gromov-Hausdorff sense. 

\end{definition}

By taking a diagonal sequence, it is obvious that $\mathcal{S}_X$ is closed.  The following lemma  is the pointed version of Theorem 4.1 in \cite{RZ} due to Rong-Zhang, establishing a local isometry and global homeomorphism between $\mathcal{R}_X$ and $X_\infty\setminus \mathcal{S}_X$. 

\begin{lemma} \label{lmlochom} There exists a continuous surjection
$$f: \overline{ (  \mathcal{R}_X,  g_\infty)} \rightarrow (X_\infty, d_\infty)$$
such that 
$$ f: (\mathcal{R}_X,  g_\infty) \rightarrow (X_\infty\setminus \mathcal{S}_X, d_\infty)$$
is a homeomorphism and a local isometry, where $\overline{ (  \mathcal{R}_X, g_\infty)}$ is the metric completion of $ \mathcal{R}_X$ with respect to the smooth limiting metric $g_\infty$.

\end{lemma}

Lemma \ref{lmlochom} immediately implies the following corollary because all tangent cones at each point in $X_\infty\setminus \mathcal{S}_X$ are the flat $\mathbb{C}^n$.

\begin{corollary}

$$\mathcal{S} \subset \mathcal{S}_X.$$

\end{corollary}

We then want to show that $\mathcal{S}_X \subset \mathcal{S}$. We look at the parabolic Monge-Ampere equation corresponding to the normalized K\"ahler-Ricci flow
$$\ddt{\psi_k(t)}= \log \frac{ (\chi + e^{-t} (\omega_k - \chi) + \ddbar \psi_k(t) )^n}{\Omega} - \psi_k(t), ~ \psi_k(0)=0.$$ 
%
%Let $A$ be an ample divisor such that $\omega_0 \in [A]$. Let $\sigma_A$ be the defining section for $A$ and $h_A$ a smooth hermitian metric on $A$ such that $Ric(h_A) = \omega_0$. Then on $X\setminus A$, we have 
%
\begin{eqnarray*}
\ddt{g_k(t) } &=&  -e^{-t} (\omega_k  - \chi) + \ddbar \psi_k (t) \\
&=& \ddbar\left\{  \psi_k(t) -  e^{-t}\varphi_k  -  e^{-t-k} (1-e^{-k})^{-1}  \epsilon_0 \log |\sigma_D|^2_{h_{D, \epsilon_0}} \right\}. 
\end{eqnarray*}
Therefore  on $X\setminus A$, we have 
\begin{equation}\label{ricpo}
Ric(g_k(1)) - g_k(1) = \ddt{g_k(1)}  =  \ddbar\left\{  \psi_k(1) -  e^{-1} \varphi_k  -  e^{-k-1} (1-e^{-k})^{-1} \epsilon_0 \log |\sigma_D|^2_{h_{D, \epsilon_0}} \right\}. 
\end{equation}

The following theorem   is due to Demailly \cite{D} for solving global $\dbar$-equation on pseudo effective line bundles on projective manifolds.

\begin{theorem} \label{demailly} Let $X$ be an $n$-dimensional  projective manifold equipped with a smooth K\"ahler metric $\omega$. Let $L$ be a holomorphic line bundle over $X$ equipped with a possibly singular hermitian metric $h$ such that $Ric(h) = -\ddbar \log h \geq \delta \omega$ in current sense for some $\delta>0$. Then for every  $L$-valued $(n,1)$-form $\tau$ satisfying 
$$ \dbar \tau =0, ~ \int_X |\tau|^2_{h, \omega} ~\omega^n<\infty,$$
where $|\tau|_{h, \omega}^2 = tr_{\omega} \left( \frac{ h \tau \overline\tau}{\omega^n} \right)$, 
there exists an  $L$-valued $(n, 0)$-form $u$ such that $\dbar u = \tau$ and 
\begin{equation}
\int_X |u|^2_h ~\omega^n \leq \frac{1}{2\pi \delta} \int_X |\tau|_{h, \omega}^2 ~\omega^n.
\end{equation}

\end{theorem}

Now we can prove the main result of the section. 

\begin{lemma}    \label{coinreg}
$$\mathcal{S}_X =  \mathcal{S}.$$

\end{lemma}

\begin{proof} Suppose not. There exist a sequence of points $q_k \in X\setminus \mathcal{R}_X$  such that $q_k$ converges to $q_\infty\in \mathcal{R}$. Then there exists a sufficiently small $r_0>0$ such that the limiting metric $d_\infty$ induces a  smooth K\"ahler-Einstein metric $g'_\infty$ on $B_{d_\infty}(p_\infty, 3r_0)$. 

Using the modified Perelman's pseudolocaity theorem in \cite{TW}, we know that $B_{g_k(1)}(q_k, 3 r_0)$ converges smoothly to $B_{d_\infty}(q_k, 3 r_0)$. More precisely, for each $k$, there exists a diffeomorphism $\phi_k : B_{d_\infty}(q_\infty, 3r_0) \rightarrow X$ such that 
$$\phi_k(q_\infty) = q_k, ~   \| \phi_k^*~ g_k(1) - g_\infty\|_{C^l(g_\infty)} \rightarrow 0, \|\phi_k^* I_k - I_\infty\|_{C^l({g_\infty} )} \rightarrow 0$$
for any fixed $l>0$, where $I_k$ and $I_\infty$ are the complex structures on $B_{g_{k(1)}}(q_k, 3r_0)$ and $B_{d_\infty}(q_\infty, 3 r_0)$.

Therefore we can assume then that the curvature of $g_k$ on $B_{g_k(1)}(q_k, 2r_0)$ is uniformly bounded and the injectivity radius of $g_k$ at $q_k$ is strictly greater than $2r_0$, for all $k$.  We can can pick complex coordinates $z^{(k)}= (z_1^{(k)}, ..., z_n^{(k)})$ on $B_{g_k(1)} (q_k, r_0)$ and  there exists $K>0$ independent of $k$ such that on $B_{g_k(1)}(q_k, r_0)$
$$q_k= 0, ~ K^{-1} d_{g_k(1)} \leq d_{\mathbb{C}^n, k} \leq K d_{g_k(1)},$$
where $d_{\mathbb{C}^n, k}$ is the Euclidean metric induced by $z^{(k)}$.

Let 
$$\rho_k = \left| z^{(k)}\right|^2, ~ \left(z^{(k)} \right)^\mu  = \Pi_{i=1}^n  \left(z^{(k)}_i \right)^{\mu_i}, ~  \mu = (\mu_1, ..., \mu_n). $$
Let $\eta$ be a smooth cut-off function with $\eta(x)=1$ for $t<1/2$ and $\eta(x)=0$ for $x\geq 1$. 
We now construct the weight $\Psi_k$ by
$$\Psi_k =  (|  \mu | +10) \eta\left(\frac{\rho_k^2}{100K}\right) \log (\rho_k^2) + \psi_k (1) - (e^{-1} - e^{-k-1}) \varphi_k   - e^{-k-1} \epsilon_0 \log |\sigma_D |^2_{h_{D, \epsilon_0} } .$$
We then define
$$h_k = h_\chi e^{-\psi_k(1) - (e^{-1} - e^{-k-1})\ddbar \varphi_k - e^{-k-1} \epsilon_0 \log |\sigma_D|^2_{h_{D, \epsilon_0}}}, ~ g_k(1) = Ric(h_k)  . $$
Then there exists $m_0$ such that for $m\geq m_0$, we have 
\begin{eqnarray*}
&&m Ric(h_k) + Ric(g_k(1)) + \ddbar \Psi_k\\
 &=& (m-1)g_k(1) + \ddbar  \left( (|\mu| +10) \eta\left(\frac{\rho_k^2}{100K}\right) \log (\rho_k^2)\right)\\
&\geq &  \frac{m}{2} ~g_k(1).
\end{eqnarray*}
Now we fix such an $m$ for all $k$ and we consider 
$$\eta_{k,  \mu}  = \dbar \left( \eta\left( \frac{\rho_k }{100K}\right) \left( z^{(k)} \right)^{ \mu} \right). $$
$\eta_{k,  \mu}$ is smooth on $X$ and hence it is $L^2$-integrable with respect to $(h_k)^m e^{-\Psi_k}$ for sufficiently large $k$.

Then we can solve for $u_k $ satisfying%
$$\dbar u_k = \eta_{k, \mu}, ~ \int_X |u_k|^2_{ h_k^me^{-\Psi_k}}  (\omega_k(1))^n < \infty. $$
This forces $u_k$ to be holomorphic near $q_k$ and vanishes to order $| \mu |$. This implies that
$\eta(\rho_k / M) (z^{(k)})^{  \mu} - u_k$ is a holomorphic section of $m K_X$ generating $\mu$-jet at $q_k$ and so $q_k\in \mathcal{R}_X$ for sufficiently large $k$. This is contradiction.

\end{proof}

We remark that we cannot apply the proof of Lemma 4.4 in \cite{S3} to prove Lemma \ref{coinreg} since we do not have a contraction morphism from $X$ to its canonical model. Immediately, we have the following corollary.
\begin{corollary} 

The metric completion of $(\mathcal{R}_X, g_\infty)$ is isomorphic to $(X_\infty, d_\infty)$. In particular, $$\mathcal{R}_X = \mathcal{R}.$$

\end{corollary}
Therefore, the metric completion of $(\mathcal{R}_X, g_\infty)$ is isomorphic to $(X_\infty, d_\infty)$ and there is an isomorphism between $(\mathcal{R}_X, g_\infty)$ and $(\mathcal{R}, g_{KE})$. We can now simply identify $\mathcal{R}$ and $g_{KE}$ with $\mathcal{R}_X$ and $g_\infty$. We also have the following technical corollary.

\begin{corollary} Let $ D'$ be any effective divisor such that $L- \epsilon  D'$ is ample for some $\epsilon>0$. Then $X\setminus D' \subset \mathcal{R}_X$. 

\end{corollary}

\section{$L^2$-estimates}

We now pick a base point $p$ in $\mathcal{R}_X$ as in Lemma \ref{locesk}. $(X, p, g_k)$ converges to the metric length space of $(X_\infty, p_\infty, d_\infty)$, where $g_k$ is defined in (\ref{maeq}). We assume that 
$$diam_{g_k}(X) \rightarrow \infty$$
because the proof for Theorem \ref{main1} is much simpler if $diam_{g_k}(X)$ is uniformly bounded for all $k$ and so  the limiting space $(X, d_\infty)$ being a compact metric space. 

Let $B_\infty(r)$ be the geodesic ball in $(X_\infty, d_\infty)$ centered at $p_\infty $ with radius $r>0$. Let $B_k(r)$ be the geodesic ball in $(X, g_k)$ centered at $p$ of radius $r$. Then $B_k(r)$ converges to $B_\infty(r)$ in Gromov-Hausdorff topology as $k \rightarrow \infty$.  We will derive local $L^2$-estimates on each $B_\infty(r)$ for all $r>0$.

Let $g_{KE}=g_\infty$ be the limiting smooth K\"ahler-Einstein metric as the smooth limit of $g_k$ on $\mathcal{R}$ and $h_{KE} = h_\chi e^{-\varphi_{KE}} = (\omega_{KE})^{-n}$ be the hermitian metric on $K_X$ on $\mathcal{R}$, where $\varphi_{KE} =\varphi_\infty$ as the limiting K\"ahler potential in section 2 and $\omega_{KE}$ is the K\"ahler-Einstein form associated to $g_{KE}$.
Each $\sigma \in H^0(X, mK_X)$ can be defined on $\mathcal{R}$ on $X_\infty$ and Corollary \ref{co22} and Corollary \ref{co23} imply that  
$$\sup_{\mathcal{R}} |\sigma |^2_{(h_{KE}) ^m} < \infty, ~ \sup_{\mathcal{R}} |\nabla \sigma |^2_{g_{KE}, (h_{KE})^m} < \infty. $$ Immediately we have the following corollary because $\mathcal{R}$ is an open dense convex subset of $(X_\infty, d_\infty)$.

\begin{corollary}\label{exten} For any $\sigma\in H^0(X, m K_X)$, $|\sigma|^2_{ (h_{KE})  ^m}$ extends from $\mathcal{R}$ to a Lipschitz function on $(X_\infty, d_\infty). $

\end{corollary}

For convenience, we define the following pluricanonical system for $X_\infty$.
\begin{definition} We denote $H^0(X_\infty, m K_{X_\infty})$ by the  set of all holomorphic sections in $H^0(X, mK_X)$ over $\mathcal{R}_X$.

\end{definition}

We need the cut-off functions constructed in the following lemma corresponds to Lemma 3.7 in \cite{S3}.
 
\begin{lemma} \label{app} For any $\epsilon>0$ and $K\subset\subset X \setminus D $, there exists  $\rho_\epsilon \in C^\infty(X \setminus D)$ such that 

\begin{enumerate}

\item $ 0\leq \rho_\epsilon \leq 1$, 
\smallskip

\item $Supp \rho_\epsilon \subset\subset X \setminus D$, 
\smallskip

\item $\rho_\epsilon=1$ on $K$, 
\smallskip

\item $ \int_X |\nabla \rho_\epsilon|^2 ( \omega_{KE})^n = \int_{X\setminus D}  \sqrt{-1}\partial \rho_\epsilon \wedge  \dbar \rho_\epsilon \wedge (\omega_{KE})^{n-1}< \epsilon$. 

\end{enumerate}

\end{lemma}

\begin{proof} The difference of Lemma \ref{app} from Lemma 3.7 in \cite{S3} is that a priori we do not know if the local potential $\varphi_{KE}$ of $\omega_{KE}$ is bounded. Without loss of generality, we can assume that $|\sigma|_{h_D}^2 \leq 1$. Let $\theta$ be a K\"ahler metric on $Y$ such that $\theta > Ric(h_D)$ and $[\theta]\geq -c_1(X)$. 
Let $F$ be the standard smooth cut-off function on $[0, \infty)$ with $F=1$ on $[0, 1/2]$ and $F=0$ on $[1, \infty)$. 
We then let 
$$\eta_\epsilon = \max ( \log |\sigma|^2_{h_D}, \log \epsilon) . $$
For sufficiently small $\epsilon$, we have $-\log \epsilon \leq \eta_\epsilon \leq 0$. 
Then obviously, $\eta_\epsilon \in PSH(X, \theta) \cap C^0(X)$.
Now we let 
$$\rho_\epsilon = F\left(\frac{\eta_\epsilon}{\log \epsilon} \right).$$
Then  $\rho_\epsilon =1 $ on $K$ if $\epsilon$ is sufficiently small. We first notice that for fixed $\epsilon>0$
$$\int_X \sqrt{-1} \partial \rho_\epsilon \wedge \dbar \rho_\epsilon \wedge (\omega_{KE})^{n-1} = \lim_{k\rightarrow \infty}   \int_X \sqrt{-1} \partial \rho_\epsilon \wedge \dbar \rho_\epsilon \wedge (\omega_k)^{n-1}$$
from the smooth uniform convergence of $\omega_k$ on any compact subset of $X\setminus D$,  where $\omega_k$ is the almost K\"ahler-Einstein metric defined in equation (\ref{maeq}).
Straightforward calculations give
\begin{eqnarray*}
&&\int_X \sqrt{-1} \partial \rho_\epsilon \wedge \dbar \rho_\epsilon \wedge \omega_k^{n-1}  \\
&=&(\log \epsilon)^{-2}  \int_X (F')^2 \sqrt{-1}\partial \eta_\epsilon \wedge \dbar \eta_\epsilon \wedge \omega_k^{n-1}\\
&\leq& C (\log \epsilon)^{-2} \int_X (-\eta_\epsilon) \ddbar \eta_\epsilon \wedge \omega_k^{n-1} \\
&\leq& C(\log \epsilon)^{-2} \int_X (-\eta_\epsilon)(\theta+ \ddbar \eta_\epsilon) \wedge \omega_k^{n-1}
+ C(\log \epsilon)^{-2} \int_X \eta_\epsilon ~ \theta \wedge \omega_k^{n-1}\\
&\leq& C(-\log \epsilon)^{-1} \int_X (\theta+ \ddbar \eta_\epsilon) \wedge \omega^{n-1} \\
&\leq& C(-\log \epsilon)^{-1} [\theta]^n
\end{eqnarray*}
where $C$ only depends does not depend on $\epsilon$ and $k$ as the class $[\omega_k]$ is uniformly bounded for all $k$. Hence
$$\lim_{\epsilon\rightarrow 0} \int_X \sqrt{-1} \partial \rho_\epsilon \wedge \dbar \rho_\epsilon \wedge (\omega_{KE})^{n-1} = 0.$$
Therefore we obtain  $\rho_\epsilon \in C^0(X)$ satisfying the conditions in the lemma. The lemma is then proved by smoothing $\rho_\epsilon$ on $Supp~ \rho_\epsilon \setminus K$.

\end{proof}

We  denote  $  \| \sigma \|_{L^{2, \sharp}(B_\infty(2R))}$, $\|\sigma\|_{L^{\infty, \sharp}(B_\infty(R))}$ and  $\|\nabla \sigma \|_{L^{\infty, \sharp}(B_\infty(R)) } $ for their norms with respect to $(h_{KE})^m$ and $mg_{KE}$ as in Section 4.3 in \cite{S3}. 
We can then apply exactly the same argument in Section 4.3 of \cite{S3} thanks to Corollary \ref{co22}, Corollary \ref{co23} and  Lemma \ref{app}.

\begin{lemma} \label{ll31}   For any $R>0$, there exists  $K_R >0$ such that if  $\sigma \in H^0(X_\infty, m K_{X_\infty})$ for $m\geq 1$, then 
\begin{equation}\label{lll1}
\|\sigma\|_{L^{\infty, \sharp}(B_\infty(R))} \leq K_R \| \sigma\|_{L^{2, \sharp}(B_\infty(2R) )}
\end{equation}
\begin{equation}\label{lll2}
 \|\nabla \sigma \|_{L^{\infty, \sharp}(B_\infty(R)) } \leq K_R \| \sigma \|_{L^{2, \sharp}(B_\infty(2R))} . 
\end{equation}

\end{lemma}

The following proposition is essentially Proposition 4.5 in \cite{S3} for solving the following $\dbar$-equation. 

\begin{proposition} \label{lll3} Let $X$ be a smooth minimal manifold of general type. Let $\omega = m\omega_{KE} \in - mc_1(X)$ and $h = (h_{KE})^m$ be a hermitian metric on $(K_X)^m $  for any $m\geq 2$. Then for any smooth $m K_X $-valued $(0,1)$-form $\tau$ satisfying

\begin{enumerate}

\item $\dbar \tau =0$,  

\item $Supp ~\tau \subset \subset \mathcal{R}_X$, 

\end{enumerate}
there exists an $m K_X $-valued section $u$ such that $\dbar u = \tau$ and $$ \int_X |u|^2_h ~\omega^n \leq \frac{1}{\pi} \int_X |\tau|^2_{h, \omega}~ \omega^n. $$

\end{proposition}

\begin{proof} The proof is almost identical to that of Proposition 3.5 and Proposition 4.5 in \cite{S3}. The only difference is that a priori we do not know if $\varphi_{KE}$ is bounded. Following the proof of Proposition 4.5 in \cite{S3}, we consider the Monge-Ampere equation for sufficiently small $0<\epsilon<\epsilon_0$
$$\left(\chi -\epsilon \epsilon_0 Ric(h_{D, \epsilon_0}) + \ddbar \varphi_\epsilon \right) ^n = e^{(1+ \epsilon)\varphi_\epsilon} \Omega.$$
and let
$$h_\epsilon = e^{- \left(\varphi_\epsilon + \epsilon\epsilon_0 \log |\sigma_D|^2_{h_{D, \epsilon_0}} \right) } h_\chi, ~\omega_\epsilon = \chi - \epsilon\epsilon_0 Ric(h_{D, \epsilon_0}) + \ddbar \varphi_\epsilon, ~ \alpha_\epsilon  =\chi - (1+\epsilon)\epsilon_0 Ric(h_{D, \epsilon_0}). $$
We now can apply the maximum principle with the key observation that there exists $C>0$ such that for all $\epsilon\in (0, \epsilon_0)$
\begin{equation}
\varphi_{KE} \geq \varphi_\epsilon + \epsilon \epsilon_0 \log |\sigma_D|^2_{ h_{D, \epsilon_0}}- C. 
\end{equation}
Then for sufficiently small $\epsilon$, we have
$$Ric(h_\epsilon) = \omega_\epsilon, ~Ric(\omega_\epsilon) = -(1+\epsilon) \omega_\epsilon + \epsilon \alpha_\epsilon \geq -(1+ \epsilon)\omega_\epsilon. $$ Since $\varphi_\epsilon$ converges to $\varphi_\infty$ as $\epsilon \rightarrow 0$,  we can apply  Theorem \ref{demailly} by writing $mK_{X} = (m-1)K_{X} + K_X$. We then can proceed as in the proof of Proposition 3.5 in \cite{S3}.

\end{proof}

\section{Local freeness for  the limiting metric space}

We will apply the same argument as in Section 3.4 and Section 4.4 \cite{S3}.  For completeness, we include  the definition of the $H$-condition introduced by Donaldson-Sun \cite{DS} and the sketch of proof for our main result in Proposition \ref{inject}. 

\begin{definition} We consider the follow data $(p_*, D, U, J, L , g, h, A)$ satisfying 

\begin{enumerate}

\item  $(p_*, U, J, g)$  is an open bounded K\"ahler manifold with a complex structure $J$, a K\"ahler metric $g$ and a base point $p_* \in D\subset \subset U$ for an open set $D$, 

\item $L \rightarrow U$ is a hermitian line bundle equipped with a hermitian metric $h$ and $A$ is the connection induced by the hermitian metric $h$ on $L$, with  its curvature $\Omega(A) = g$.

\end{enumerate}

\noindent The data $(p_*, D, U, L, J, g, h, A)$ is said to satisfy the  $H$-condition if there exist $C>0$ and  a compactly supported smooth section $\sigma: U\rightarrow L$   satisfying 

\begin{enumerate}

\item [$H_1$:]   $\|\sigma \|_{L^2(U)} < (2\pi)^{n/2}$,

\item [$H_2$:]  $|\sigma(p_*) | >3/4$,

\item [$H_3$:] for any holomorphic section $\tau$ of $L$ over a neighborhood of $\overline{D}$,
$$|\tau(p_*)| \leq C (\| \dbar \tau \|_{L^{2n+1}(D)} + ||\tau||_{L^2(D)} ),$$

\item [$H_4$:]  $\|\dbar \sigma \|_{L^2(U)} < \min \left( \frac{1}{8\sqrt{2} C},  10^{-20} \right)$,

\item [$H_5$:]  $||\dbar \sigma ||_{L^{2n+1}(D)} \leq \frac{1}{8C}$.

\end{enumerate}

\end{definition}

Here all the norms are taken with respect to $h$ and $g$. The constant $C$ in the $H$ condition depends on the choice $(p_*, D, U, J, L, g, h)$. 

Fix any point $p$ on $X_\infty$,  $(X_\infty, p, m\omega_\infty)$ converges in pointed Gromov-Hausdorff topology to a tangent cone $C(Y)$ over the cross section $Y$, where $\omega_\infty = \omega_{KE}$. We still use $p$ for the vertex of $C(Y)$.  We write $Y_{reg}$ and $Y_{sing}$ the regular and singular part of $Y$. $Y_{sing}$ has Hausdorff dimension strictly less than $2n-3$.  $C(Y_{reg})\setminus \{p \}$ has a natural complex structure induced from the Gromov-Hausdorff limit and the cone metric $g_C$ on $C(Y)$ is given by
$$g_C = \frac{1}{2} \ddbar r^2 ,$$
where $r$ is the distance function for any point $z\in C(Y)$ to $p$. We can also write the cone metric $g_C= \frac{1}{2} \ddbar |z|^2$.
One considers the trivial line bundle $L_C$ on $C(Y)$ equipped with the connection $A_C$ whose curvature coincides with $g_C$. The curvature of the hermitian metric defined by $h_C=e^{-|z|^2/2}$ is $g_C$.  $1$ is a global section of $L_C$ with its norm equal to  $e^{-|z|^2/2}$ with respect to $h_C$. 
The following lemma is due to \cite{DS}. 

\begin{lemma} \label{hcon} Let $p_* \in C(Y_{reg})$. If $3/4 <e^{ - |p_*|^2}  <1$, then for any $\epsilon>0$, there exists $U\subset \subset C(Y_{reg})\setminus \{p\}$ and an open neighborhood $D\subset\subset U$ of $p_*$ such that 
 $(p_*, D, U, L_C, J_C, g_C, h_C, A_C)$ satisfies the $H$-condition.

\end{lemma}

By the construction in \cite{DS}, we can always assume that both $D$ and $U$ are a product in $C(Y_{reg})\setminus \{p \}$, i.e., there exist $D_Y$ and $U_Y \subset Y_{reg}$ such that $D=\{ z=(y, r) \in C(Y)~|~y \in D_Y, ~r\in (r_D, R_D)\}$ and  $U=\{ z=(y, r) \in C(Y)~|~y \in U_Y, ~r\in (r_U, R_U)\}$. Suppose $(p_*, D, U, L_C, J_C, g_C, h_C , A_C)$ satisfies the $H$-condition from Lemma \ref{hcon}. For any $m\in \mathbb{Z}^+$, we can define 
\begin{equation}\label{rescue}
U(m) = \{ z=(y, r)\in C(Y)~|~y\in U_Y, ~ r\in (m^{-1/2} r_U, m^{1/2} R_U) \} 
\end{equation}
and
$\mu_m: U \rightarrow U(m)$ by
$$ \mu_m(z) = m^{-1/2} z. $$

The following proposition from \cite{DS} establishes the stability of the $H$-condition for perturbation of the curvature and the complex structure. 

\begin{proposition}  \label{hcon2} Suppose $(p_*, D, U, J_C, L_C, g_C, h_C, A_C )$ constructed as above in Lemma \ref{hcon} satisfies the $H$-condition. There exist $\epsilon>0$ and $m \in \mathbb{Z^+}$ such that for any collection of data $(p_*, D, U, J, g, h, A)$ if 
$$|| g - g_C ||_{C^0(U(m)) } + ||J-J_C ||_{C^0(U(m)) } < \epsilon, $$
then for some $1\leq l \leq m$, 
$$(p_*, D, U, \mu_l^*J, \mu_l^*L, \mu_l^*g, \mu_l^*h, \mu_l^*A)  $$
satisfies the $H$-condition.

\end{proposition}

Fix any point $p$, we can assume that $(X_\infty, p, m_v g_\infty)$ converges to a tangent cone $C(Y)$ for some sequence $m_v$ in pointed Gromov-Hausdorff topology. In particular, on the regular part of $C(Y)$, the convergence is locally $C^{2, \alpha}$ and the metrics $m_v g_\infty$ converge locally in $C^{1, \alpha}$.  Fix any open set $U\subset\subset C(Y_{reg})\setminus \{p\}$, This  would induce embeddings $\chi_{m_v}: U \rightarrow \mathcal{R}= (X_\infty)_{reg}$. Let $g_{m_v}$ be the pullback metric of $g_\infty$ on $(X_\infty)_{reg}$ and $J_{m_v}$ be the pullback complex structure. The following lemma follows from the convergence of $(X_\infty, p, m_v g_\infty)$.

\begin{lemma} \label{closen} There exists $v$ such that one can find an embedding $\chi_{m_v}$such that 

\begin{enumerate}

\item $2^{-1} |z| \leq (m_v)^{1/2} d_\infty(p, \chi_{m_v}(z)) \leq 2   |z|$,

\smallskip 

\item $|| \chi_{m_v}^*(m_v g_\infty ) - g_C ||_{C^0(U)} + ||\chi^*_{m_v} J_\infty - J_0||_{C^0(U)} \leq \epsilon $,
\smallskip

\end{enumerate}
where $d_\infty$ is the metric on $X_\infty$ induced from $g_\infty$.

\end{lemma}

The following is the main result in this section. 

\begin{proposition}\label{inject}   For any point $q \in X_\infty$,  there exist $m\in \mathbb{Z}^+$ and  a holomorphic section $\sigma \in H^0(X_\infty, m K_{X_\infty})$ such that 
$$|\sigma|_{(h_\infty)^m}^2(q)\neq 0, $$
where $h_\infty= h_{KE}= h_\chi e^{-\varphi_{KE}}.$

\end{proposition}

\begin{proof}  The proposition can be proved  exactly the same argument for Proposition 3.9 and Proposition 4.6 in \cite{S3}. 
\end{proof}

We remark that from Proposition \ref{inject}, we cannot conclude $\sigma(q)\neq 0$, because $\varphi_{KE}$ might tends to $-\infty$ near $q$.

\section{Global freeness}

We will complete the proof for Theorem \ref{main1} in this section by proving freeness of $m K_X$ at any fixed point $q\in X$ for some sufficiently large $m$ depending on $q$.

First we want to show that for any fixed point $q$ on $D$, the distance from $q$ to a fixed point $p\in \mathcal{R}_X$ is uniformly bounded for all $g_k$. We  consider the a log resolution 
$$\pi_1: Z \rightarrow X$$
 such that $\pi_1^{-1} (D)$ is   a union of smooth divisors with simple normal crossings and there exists $O$ in the smooth part of $\pi_1^{-1}(D)$ with $\pi_1(O)=q$. We then blow up $Z$ at $O$ with 
 $$\pi_2: \tilde X \rightarrow Z.$$ Let 
 $$\tilde \pi = \pi_1\circ\pi_2: \tilde X \rightarrow X. $$
We have the following adjunction formula because $X$ is smooth
$$K_{\tilde X} = \tilde\pi^* K_X + (n-1) E  + F, ~F=\sum_{j=1}^J a_j [F_j], $$
where $(n-1)E+ F$ is the exceptional divisor of $\tilde \pi$, $F_j$ are effective prime smooth divisors on $\tilde X$ with $ a_j > 0$ for $j=1, ..., J$, $E$ is the exceptional divisor of $\pi_2$ isomorphic to $\mathbb{CP}^{n-1}$.

Since $\tilde \pi^* K_X $ is big and nef, by Kodaira's lemma, there exists an effective divisor $\tilde D$ such that its support coincides with the support of the exceptional divisors of $\tilde\pi$ and 
$$\tilde \pi ^*K_X - \epsilon \tilde D$$ is ample for all sufficiently small $\epsilon>0$. We can always assume that $\tilde D = \tilde D' + E$, where the effective divisor $D'$ does not contain $E$.   Let $\sigma_E$, $\sigma_F$ and $\sigma_{\tilde D}$  be the defining sections of $E$, $F$ and  $\tilde D$. Here we consider $\sigma_E$, $\sigma_F$ and $\sigma_{\tilde D}$ be the multivalued holomorphic sections which become  global holomorphic sections after taking some power. Let $h_E, h_F, h_{\tilde D} $ be smooth hermitian metrics on the line bundles associated to $E$, $F$ and $\tilde D$ such that 
$$ (\tilde\pi)^* \Omega = |\sigma_E|^{2(n-1)}_{h_E} |\sigma_F |^2_{h_F} \tilde \Omega, ~\tilde \chi - \epsilon_0 Ric(h_{\tilde D, \epsilon_0}) >0 $$ for a smooth volume form $\tilde \Omega$ on $\tilde X$ and for some fixed sufficiently small $\epsilon_0>0$, where $\tilde \chi = ( \tilde \pi)^*\chi$.

Let $\tilde \omega$ be a fixed smooth K\"ahler form on $\tilde X$. Then the K\"ahler-Einstein equation lifted to $\tilde X$ is equivalent to the following degenerate Monge-Ampere equation
$$(\tilde\chi + \ddbar \tilde \varphi_{KE})^n = e^{\tilde \varphi_{KE} } (\tilde \pi)^*  \Omega, $$
where $\tilde\varphi_{KE} = (\tilde\pi)^*\varphi_{KE}.$
We consider the following family of Monge-Ampere equations
\begin{equation}\label{4365}
((1-e^{-k}) \tilde\chi + e^{-k} \tilde \pi^* \omega_0 + \epsilon \tilde\omega + (1-e^{-k})\ddbar \tilde \varphi_{k,\epsilon} )^n = e^{\tilde\varphi_{k,\epsilon} }  (|\sigma_E|^{2(n-1)}_{h_E} +\epsilon) (|\sigma_F|^2_{h_F}+\epsilon) \tilde \Omega, 
\end{equation}
where $\omega_0$ is the K\"ahler form defined as in equation (\ref{refmc}). 
Let 
$$\tilde\omega_{k, \epsilon }= (1- e^{-k}) \tilde \chi + e^{-k}  \omega_0 + \epsilon \tilde\omega + (1-e^{-k}) \ddbar \tilde \varphi_{k,\epsilon}. $$ By Yau's theorem, equation (\ref{4365}) always admits a unique smooth solution $\tilde\varphi_{k,\epsilon}$ for all sufficiently small $\epsilon>0$. 

The following lemma corresponds to Lemma 4.9 in \cite{S3}. 

\begin{lemma} There exists $A>0$ such that for all $k\in \mathbb{Z}^+$ and $\epsilon \in (0, 1)$,

$$Ric(\tilde\omega_{k,\epsilon} ) \leq - (1- e^{-k}) \tilde\omega_{k,\epsilon} + A \tilde \omega. $$

\end{lemma}

\begin{lemma}  \label{483} Let  $\tilde\varphi_{k,\epsilon}$ be the smooth solution for the  equation (\ref{4365}) for $k\in \mathbb{Z}^+$ and $\epsilon\in (0, 1)$. Then there exists $C>0$ and 
for any $\delta>0$, there exists $C_\delta>0$ such that for  all $k\in \mathbb{Z}^+$ and $\epsilon\in (0,1)$
\begin{equation}\label{c06}
-\delta \log |\sigma_{\tilde D}|^2_{h_{\tilde D}} - C_\delta \leq \tilde \varphi_{k,\epsilon} \leq C
\end{equation}
and there exist $\lambda$, $C>0$ such that for all $k>0$ and $\epsilon\in (0,1)$, 
\begin{equation}\label{c26}
\tilde \omega_{k,\epsilon} \leq C   \left( |\sigma_{\tilde D}|_{h_{ \tilde D}}^{2} \right)^{-\lambda} \tilde \omega.
\end{equation}
Furthermore, $\tilde \varphi_{k, \epsilon}$ converges to $\tilde\varphi_{KE}$  smoothly on $\tilde X \setminus   \tilde D$ as $k^{-1}, \epsilon \rightarrow 0$.

\end{lemma}

\begin{proof} The uniform upper bound for $\tilde \varphi_{k, \epsilon}$ follows from the mean value inequality for the plurisubharmonic function combined with the Jensen inequality and the uniform bound for $\int_{\tilde X} e^{\tilde \varphi_{k, \epsilon} } (|\sigma_E|^{2(n-1)} + \epsilon) (|\sigma_D|^2)_{h_D}+\epsilon) \tilde \Omega$. To prove the lower bound for $\tilde \varphi_{k, \epsilon}$, we consider
$$\tilde \varphi_{k, \epsilon, \delta} = \tilde \varphi_{k, \epsilon} - \delta \log |\sigma_{\tilde D}|^2_{h_{\tilde D, \delta}}. $$
Then $\tilde \varphi_{k, \epsilon, \delta}$ satisfies on $\tilde X \setminus \tilde D$
$$ (\tilde\chi-\delta Ric(h_{\tilde D,\delta}) + k^{-1} \omega_0 + \epsilon \tilde\omega + \ddbar \tilde \varphi_{k,\epsilon,\delta} )^n = e^{\tilde\varphi_{k,\epsilon, \delta} }  |\sigma_{\tilde D}|^{2\delta}_{h_{\tilde D}} (|\sigma_E|^{2(n-1)}_{h_E} +\epsilon) (|\sigma_F|^2_{h_F}+\epsilon) \tilde \Omega. $$
Let $\psi_{k, \epsilon, \delta} \in PSH(\tilde X, \tilde \chi - \delta Ric(h_{\tilde D}) )\cap C^0(\tilde X)$ solves
$$ (\tilde\chi-\delta Ric(h_{\tilde D,\delta}) + k^{-1} \omega_0 + \epsilon \tilde\omega + \ddbar \psi_{k, \epsilon,\delta} )^n = e^{\psi_{k,\epsilon, \delta} } (|\sigma_E|^{2(n-1)}_{h_E} +\epsilon) (|\sigma_F|^2_{h_F}+\epsilon) \tilde \Omega.$$
Then there exists $C=C(\delta)$ independent of $k, \epsilon$ such that  
$$\|\psi_{k, \epsilon, \delta}\|_{L^\infty(\tilde X)} \leq C.$$
Then the low bound for $\tilde\varphi_{k,\epsilon, \delta} $ follows immediately from the maximum principle. The estimate (\ref{c26}) can be proved by standard maximum principle using Tsuji's trick. 

\end{proof}

Let $B_O$ be a sufficiently small Euclidean ball on $Z$ centered at $O$  and let $\tilde B_O = \pi_2^{-1}(B_O)$ in $\tilde X$. The support of $\tilde D'$ and $\tilde F = (\pi_2)^{-1} (F) - E$, the proper transformation of $F$,  lie in the subvariety defined by  $w =0$ for a holomorphic function $w$.

Lemma \ref{483} immediately implies the following claim. 

\begin{lemma} \label{co43} Let $\tilde B_O= \pi_2^{-1}(B_O)$. There exist $\lambda, C>0$ such that for all $k\in \mathbb{Z}^+$ and $\epsilon \in (0,1)$, 
\begin{equation}
\tilde \omega_{k,\epsilon} |_{\partial \tilde B_O}  \leq  C \left(  |w|^{2\lambda} ~\tilde \omega \right)|_{\partial \tilde B_O}. 
\end{equation}

\end{lemma}

The following lemma is purely local and it corresponds to Lemma 4.11 in \cite{S3}. 

\begin{lemma}

Let $\hat\omega$ be the smooth closed nonnegative closed $(1,1)$-form as the pullback of the Euclidean metric $\sqrt{-1} \sum_{j=1}^n dz_j \wedge d\bar z_{j}$ on $B_O$. $\hat\omega$ is K\"ahler on $\tilde B_O \setminus E$. There exist $C>0$, sufficiently small $\epsilon_0>0$ and a smooth hermitian metric $h_E$ on $E$ such that on $\tilde B_O$, 
\begin{equation}\label{locba}
C^{-1} \hat\omega \leq \tilde\omega \leq C |\sigma_E|^2_{h_E}  \hat\omega, 
\end{equation}
\begin{equation}\label{locbaa} %
\tilde \chi - \epsilon_0 Ric(h_E) > C^{-1}\tilde \omega. 
\end{equation}

\end{lemma}

The following is the main estimate in this section and the proof follows from the proof of Proposition 4.7 in \cite{S3}. The only difference is that in our situation $\tilde\varphi_{k, \epsilon}$ is not uniformly bounded in $L^\infty$, however, the estimate (\ref{c06}) suffices to achieve the same estimate  (\ref{c261}) by increasing $\lambda$. 

\begin{lemma} \label{digest} There exist $\nu, \lambda>0$ and $C>0$ such that for any $\epsilon\in (0, 1)$ and $k\in \mathbb{Z}^+$,    we have on $\tilde B_O$, 
\begin{equation} \label{c261}
\tilde \omega_{k, \epsilon} \leq \frac{C }{|\sigma_E|^{2(1-\nu)}_{h_E}    |w|^{2\lambda}} \tilde\omega,
\end{equation}
where $\omega_k$ is the almost K\"ahler-Einstein metric defined in (\ref{maeq}).

\end{lemma}

Lemma \ref{digest} immediately implies the following corollary by letting $\epsilon$ and then $k^{-1} \rightarrow 0$. 

\begin{corollary} \label{co44} There exist $\nu, \lambda>0$  and $C >0$ such that  on $\tilde B_O$, %
\begin{equation}
(\tilde\pi)^* \omega_{KE}  \leq \frac{C  }{|\sigma_E|^{2(1-\nu)}_{h_E}     |w|^{2\lambda}} \tilde\omega 
\end{equation}
and 
for any  $k\in \mathbb{Z}^+$, 
\begin{equation}
(\tilde\pi)^* \omega_k  \leq \frac{C  }{|\sigma_E|^{2(1-\nu)}_{h_E}     |w|^{2\lambda}} \tilde\omega.
\end{equation}

\end{corollary}

\begin{proposition} \label{arc} For any $q\in X$, there exists a smooth path $\gamma(t)$ for $t\in [0, 1]$ such that
\medskip
\begin{enumerate}

\item $\gamma(t) \in X\setminus ( D\cup \{q \})$ for $t\in [0, 1)$ and $\gamma(1)=q$,

\medskip

\item $\gamma$ is transversal to $D$, 

\medskip

\item for any $\varepsilon >0$, there exists $\delta\in (0, 1)$ such that for all $k>0$ and $t\in [1-\delta, 1]$, 
$$d_{g_k }(\gamma(t), q) \leq \varepsilon. $$

\end{enumerate}

\end{proposition}

\begin{proof} It suffices to prove the case when $q\in D$ and then the proposition is proved by picking a sufficiently small line segment $\gamma(t)$ for $t\in [0, 1]$ with $\gamma(t) \in \tilde B_O \setminus (\{w=0\} \cup E )$ for $t\in [0, 1)$ and $\gamma(1) \in (E \cap \tilde B_O ) \setminus  \{w=0\}  $. The proposition is then proved by applying estimate (\ref{c261}) as  arc length of $\gamma$ with respect to $\omega_{k, }$ is uniformly bounded for all $k$.

\end{proof}

\begin{corollary} Fix a base point $p\in \mathcal{R}_X$. Then for any point $q\in X$, there exists $A_q>0$ such that 
$$ d_{g_k}(p, q) \leq A_q. $$
Furthermore, $q$ must converge in Gromov-Hausdorff sense to a point $q_\infty\in X_\infty$.

\end{corollary}

\begin{proof} The corollary immediately follows from Corollary \ref{co44} and the line segment chosen in Proposition \ref{arc}, after letting $\epsilon \rightarrow 0$.

\end{proof}

The following is the local freeness for the pluricanonical system on $X$.

\begin{proposition} \label{locfre6} For any $q\in X$, there exist $m\in \mathbb{Z}^+$ and $\sigma  \in H^0(X, mK_X)$ such that 
$$\sigma(q) \neq 0. $$

\end{proposition}

\begin{proof} It suffices to prove the case when $q\in D$. Let $q_\infty\in X_\infty$ be the limiting point of $q$. By Proposition \ref{inject}, there exist $m>0$ and $\sigma \in H^0(X_\infty, m K_X)$ such that 
$$|\sigma|^2_{ (h_\infty)^m} (q_\infty)>1.$$ We now consider the sequence $q_j$ in the smooth path $\gamma(t)$  in Proposition  \ref{arc} such that  for all $k$,
$$d_{g_k}(q, q_j) < 1/j, ~ \lim_{j\rightarrow \infty} d_{g_0} (q_j, q) = 0. $$
Certainly $q_j\in (X, g_k)$ converges to the same point $q_j\in (X_\infty, d_\infty)$ as $k\rightarrow \infty$ by the smooth convergence of $g_k$ on $\mathcal{R}_X$ for each fixed $j$. Then 
$$d_\infty(q_j, q_\infty) = \lim_{k\rightarrow \infty} d_{g_k} (q_j, q ) < 1/j. $$
By continuity of $|\sigma|^2_{h_\infty^m}$ on $(X_\infty, d_\infty)$ from Lemma \ref{exten}, there exists $J>0$ such that for $j>J$, 
$$|\sigma|^2_{ (h_\infty)^m} (q_j) \geq 1/2 . $$
This implies that there exists $C>0$ such that for all $j>J$, 
$$ \varphi_\infty (q_j) \leq \log |\sigma|^2_{ (h_\chi)^m} (q_j) + C. $$
On the other hand, for any $\delta>0$, there exists $C_\delta>0$ such that 
$$\varphi_\infty \geq \delta \log |\sigma_D|^2_{h_D} - C_\delta. $$
Therefore for any $\delta>0$, there exists $C_\delta>0$ such that for all $j>J$, 
$$ \delta \log |\sigma_D|^2_{h_D} (q_j)  \leq \log |\sigma|^2_{h_\chi^m} (q_j) + C_\delta $$
or
$$|\sigma_D|^{2\delta}_{h_D} (q_j) \leq e^{C_\delta} |\sigma|^2_{h_\chi^m} (q_j).$$
Since $q_j$ lies in a smooth path transversal to $D$ as chosen in Proposition \ref{arc} and both $\sigma_D$ and $\sigma$ are holomorphic, this implies that $\sigma$ cannot vanish at $q$. This proves the proposition.

\end{proof}

We now can prove Theorem \ref{main1}.

\begin{theorem} Let $X$ be a smooth minimal model of general type. Then $m K_X$ is globally generated for some $m$ sufficiently large. 

\end{theorem}

\begin{proof} By Proposition \ref{locfre6}, for any $q\in X$ there exist $m_q \in \mathbb{Z}^+$ and $\sigma_q \in H^0(X, m_q K_X)$ such that $\sigma_q$ does not vanish at $q$. Then there exists an open neighborhood $U_q$ of $q$ such that $\sigma_q$ does not vanish anywhere in $U_q$. The theorem is then proved by the finite covering theorem since $X$ is compact.

\end{proof}

As an application, we prove the $L^\infty$ estimate for $\varphi_\infty$ in K\"ahler-Einstein equation (\ref{keq}). 

\begin{corollary} $$\varphi_\infty\in L^\infty(X).$$

\end{corollary}

\section{Generalizations and discussions}

\noindent{\it  7.1. Freeness for big and nef line bundles on Calabi-Yau manifolds.} Using the same argument for Theorem \ref{main1}, we can also prove Theorem \ref{main2} for the semi-ampleness of a nef and big line bundle $L$ over a Calabi-Yau manifold $X$ of dimension $n$. In fact, the proof is simpler, because one can obtain a uniform diameter bound for the approximating Ricci-flat K\"ahler metrics. We lay out the sketch of the proof and leave the details for the readers to check.  

Let $\Omega$ be the smooth volume form on $X$ such that $ \ddbar \log \Omega =0$ and $\int_X \Omega = \left( c_1(L) \right)^n.$ Let $\chi\in c_1(L)$ be a smooth closed $(1,1)$-form and $\omega_0$ a smooth K\"ahler form on $X$. We then consider the following family of Monge-Ampere equations for $k\in \mathbb{Z}^+$. 
\begin{equation}
(\chi + e^{-k} \omega_0 + \ddbar \varphi_k)^n = e^{c_k} \Omega, ~ \sup_X \varphi_k =0. 
\end{equation}
Then $Ric(\omega_k)=0$, $\lim_{k\rightarrow\infty} c_k = 0$ and $diam_{g_k}(X)$ is uniformly bounded. We can adapt the arguments in previous sections  as well as the argument in Section 3 of \cite{S3} to prove Theorem \ref{main2}.

\bigskip

\noindent{\it 7.2. Kawamata's base-point-free theorem.} Using the constructions of  conical K\"ahler-Einstein metrics, one should also be able to prove the following base-point-free theorem of Kawamata \cite{K1, K2} with an additional assumption on the bigness of the divisor $D$.

\begin{theorem} Let $X$ be a projective manifold. If $D$ is a big and nef divisor such that $a D -K_X$ is big and nef for some $a>0$, then $D$ is semi-ample. 

\end{theorem}

The assumption for $D$ being big is the noncollapsing condition to guarantee the existence of K\"ahler-Einstein current $\omega_{KE}\in c_1(a D)$ satisfying
$$Ric(\omega_{KE}) = - \omega_{KE} + \eta,$$
where $\eta\geq 0$ is a nonnegative current in $c_1(aD- K_X)$ on $X$ and it vanishes on a Zaraski open dense subset of $X$.
We will leave the more detailed discussion in our future work.

\bigskip

\noindent{\it 7.3. Toward finite generation of canonical rings and abundance conjecture.} We believe that our approach can also be applied to understand the finite generation of canonical rings, which is proved in \cite{BCHM} and \cite{Siu, Siu2}. The canonical K\"ahler-Einstein current on $X$ is constructed in \cite{S3, BEGZ} and we hope that  the scheme in our paper and in \cite{ST2} can lead to  an analytic and Riemannian geometric proof for finite generation of canonical rings for projective manifolds of general type. Another approach is the K\"ahler-Ricci flow through singularities as developed in \cite{ST3} and this should lead to deeper understanding for analytic and geometric aspects of singularities and flips in the minimal model program.

\bigskip
\bigskip

%\noindent {\bf{Acknowledgements:}} 

\end{document}